\documentclass[10pt, reqno]{amsart}
\usepackage{amsmath,amssymb,amsthm, amsfonts, amssymb, graphicx, color, array}
\usepackage[bookmarksnumbered, colorlinks, plainpages]{hyperref}
\usepackage{makecell}
\usepackage{indentfirst, amssymb,amsmath,amsthm}
\newtheorem*{theoA}{Theorem A}
\newtheorem*{theoB}{Theorem B}
\newtheorem*{theoC}{Theorem C}
\newtheorem*{theoD}{Theorem D}
\newtheorem*{theoE}{Theorem E}
\newtheorem*{theoF}{Theorem F}

\newtheorem{theo}{Theorem}[section]
\newtheorem{lem}{Lemma}[section]

\newtheorem{exm}{Example}[section]
\newtheorem{defi}{Definition}[section]
\newtheorem{rem}{Remark}[section]
\newtheorem{ques}{Question}[section]

\newcommand{\ol}{\overline}
\newcommand{\be}{\begin{equation}}
	\newcommand{\ee}{\end{equation}}
\newcommand{\beas}{\begin{eqnarray*}}
	\newcommand{\eeas}{\end{eqnarray*}}
\newcommand{\bea}{\begin{eqnarray}}
	\newcommand{\eea}{\end{eqnarray}}
\newcommand{\lra}{\longrightarrow}
\numberwithin{equation}{section}
\begin{document}
		\title[Identification of $L$-functions in the extended  ....] {\LARGE I\MakeLowercase {dentification of $L$-functions in the extended \MakeUppercase {S}elberg\\ class by preimages of finite sets}}
	
	\author{Arpita Kundu$^{*}$ and Abhijit Banerjee}	
	
	\address{ Department of Mathematics, Techno India University, West Bengal, India.}
	\email{arpitakundu.math.ku@gmail.com}
	\address{ Department of Mathematics, University of Kalyani, West Bengal 741235, India.}
	\email{abanerjee\_kal@yahoo.co.in, abanerjeekal@gmail.com}
	\maketitle
	\let\thefootnote\relax
	\footnotetext{Type set by \AmS -\LaTeX}
	\footnotetext{*Corresponding author}
	\begin{abstract}
		In 2023, Li-Du-Yi [Complex Var. Elliptic Equ., 68(10)2023, 1653-1677] established that if two $L$-functions $L_1$ and $L_2$ in the extended Selberg class  $\mathcal{S}^{\#}$ have positive
		degrees, satisfy the same functional equation with $a(1) = 1$ and share a set $S=\{c_1,c_2,c_3\},$ where $c_i\in\mathbb{C};\;i=1,2,3,$, then $L_1 = L_2.$ Subsequently, the the present authors [Complex Var. Elliptic Equ.,70(4), 716–734] generalized this result by proving that if $L_1, L_2\in S^{\#}$ have positive degrees, satisfy the same functional equation with $a(1)=1$ and share an arbitrary set $\{\alpha_1, \alpha_2,\ldots, \alpha_t\}\subset\mathbb{C}$ counting multiplicities, for any $t\geq 1$, then
		$L_1 = L_2$. In the present paper, we remove the requirement that $L_1$ and $L_2$ satisfy the same functional equation, thereby obtaining the earlier uniqueness conclusion under substantially weaker hypotheses. As a consequence, we prove that any polynomial with distinct zeros is a strong uniqueness polynomial for $L$-functions in $\mathcal{S}^{\#}$. This result strengthens and extends the work of Khoai–An [The Ramanujan J., 58(1)(2022),253-267] and certain results in  [J. Theor. Nr. Bordx., $\mathbf{36(3)}$ (2024), 967-985]. Furthermore, we take steps toward extending the results of [Math. Z., $\mathbf{272}$(2012), 1097-1102] broader setting.
		\end{abstract}
		\section{Introduction}
	\subsection*{1. Zeros as Information Carriers}
	
	A recurring theme in modern analytic number theory is that \emph{zeros encode structure}.
	Rather than treating an $L$-function merely as a formal Dirichlet series or an Euler product,  one may regard it as a global analytic object whose identity is partially—or even completely determined by partial information about its value distribution.
	This philosophy, which traces back to Riemann’s 1859 memoir, has led to a wide range of \emph{uniqueness problems} for $L$-functions.
	At a conceptual level, these problems ask the following question:
	
	\begin{quote}
		\emph{How much information about an $L$-function is encoded in the values it
			assumes—or avoids—on specified subsets of the complex plane? }
	\end{quote}
	
	The most familiar example is the Riemann zeta function, initially defined for $\Re(s)>1$ by
	\[
	\zeta(s)=\sum_{n=1}^{\infty}\frac{1}{n^s}.
	\]
	It admits a meromorphic continuation to the entire complex plane with a simple pole at $s=1$ of residue $1$.
	Moreover, the Euler product
	\[
	\zeta(s)=\prod_{p}\left(1-p^{-s}\right)^{-1}, \qquad \Re(s)>1,
	\]
	and the functional equation relating $s$ and $1-s$ illustrate how arithmetic information manifests as analytic rigidity.	Such rigidity motivates the systematic study of uniqueness phenomena for
	general $L$-functions.
	
	\subsection*{2. Structural Classes of $L$-Functions}
	
	Motivated by the analytic properties of $\zeta(s)$, Selberg introduced in 1989 the class $\mathcal{S}$ of Dirichlet series
	\[
	\mathcal{L}(s)=\sum_{n=1}^{\infty}\frac{a(n)}{n^s},
	\]
	satisfying a Ramanujan-type bound, analytic continuation, a Riemann-type functional equation, and an Euler product (see \cite{Selberg-92}).
	An $L$-function $\mathcal{L}\in\mathcal{S}$ has degree
	\[
	d_{\mathcal{L}}=2\sum_{j=1}^{K}\lambda_j,
	\]
	where $\lambda_j$ appear in the associated functional equation.
	
	While the Euler product is a defining feature of $\mathcal{S}$, from the viewpoint of value distribution it is not indispensable.
	There exist convergent Dirichlet series satisfying the first three Selberg axioms but lacking an Euler product.
	This observation led Kaczorowski and Perelli \cite{Kaczorowski_Perelli} to introduce the \emph{extended Selberg class} $\mathcal{S}^{\#}$, consisting of non-identically vanishing $L$-functions satisfying axioms (i)--(iii) of $\mathcal{S}$.
	Clearly,
	\[
	\mathcal{S}\subset\mathcal{S}^{\#}.
	\]
	
	A notable feature of $\mathcal{S}^{\#}$ is the presence of degree-zero functions that do not belong to $\mathcal{S}$.
	Such functions do not admit Euler products but nevertheless exhibit nontrivial analytic behavior, making $\mathcal{S}^{\#}$ a natural framework for investigating uniqueness problems.

\begin{table}[h]
	\centering
	\renewcommand{\arraystretch}{1.15}
	\caption{Comparison of Selberg-type classes}
	\small
	\begin{tabular}{||
			>{\centering\arraybackslash}p{1.6cm}|
			>{\centering\arraybackslash}p{1.4cm}|
			>{\centering\arraybackslash}p{1.8cm}|
			>{\centering\arraybackslash}p{1.6cm}|
			>{\centering\arraybackslash}p{1.9cm}|
			>{\centering\arraybackslash}p{2.4cm}|
			>{\centering\arraybackslash}p{1.2cm}||}
		\hline\hline
		Class
		& Dirichlet series
		& Functional equation
		& Euler product
		& Ramanujan condition
		& Analytic continuation
		& Degree \\ \hline
		\hline
		$\mathcal{S}$
		& Yes
		& Yes
		& Yes
		& Yes
		& Entire or pole at $s=1$
		& $>0$ \\ \hline
		$\mathcal{S}^{\#}$
		& Yes
		& Yes
		& Not required
		& Not required
		& Entire or pole at $s=1$
		& $\ge 0$ \\ \hline\hline
	\end{tabular}
\end{table}

	\subsection*{3. Value Distribution and Sharing Problems}
	
	To study uniqueness questions for $L$-functions, we place them within the framework of Nevanlinna theory.
	Let $\mathcal{M}(\mathbb{C})$ denote the field of meromorphic functions on the complex plane.
	We use standard notations such as the characteristic function $T(r,f)$, the proximity function $m(r,f)$, and the (reduced) counting functions $N(r,f)$ and $\overline{N}(r,f)$ as described in \cite{C.C Yang-H.X.Yi}.
	The order of a meromorphic function $f$ is defined by
	\[
	\rho(f)=\limsup_{r\to\infty}\frac{\log T(r,f)}{\log r}.
	\]
	
	Within this setting, sharing problems ask whether two non-constant meromorphic functions must coincide if they share prescribed values or sets, either counting multiplicities (CM) or ignoring them (IM).
	While classical Nevanlinna theory typically requires sharing several values for uniqueness, the presence of functional equations and growth restrictions allows far stronger conclusions for $L$-functions.
	\par A fundamental result in this direction was obtained by Steuding \cite{Steuding_Sprin(07)}, who showed that two $L$-functions in the Selberg class sharing a single finite value CM must coincide. However, providing the following counterexample, Hu-Li \cite{Hu_LI_Can(16)} pointed out that Steuding's \cite{Steuding_Sprin(07)} result cease to hold when $c = 1$.  
	\begin{exm}
		Let $\mathcal{L}_1=1+\frac{2}{2^{2s}}$ and
		$\mathcal{L}_2=1+\frac{3}{3^{2s}}$ with $a(1)=1$.
		Both functions satisfy axioms (i) and (ii) and trivially fulfill a functional
		equation of degree zero.
		They share the value $\{1\}$ CM, yet $\mathcal{L}_1\neq \mathcal{L}_2$.
	\end{exm}
	To formulate subsequent results, we introduce the following definitions.	\begin{defi}
		A polynomial $P$ is called a uniqueness polynomial for meromorphic functions if for any two non-constant meromorphic functions $f$, $g\in M(\mathbb{C})$, the condition $P(f) = P(g)$ implies $f= g$.
		\par On the other hand, for a non-zero constant $c$, if the condition $P(f)=cP(g)$ implies $f= g$, then $P$ is called a strong uniqueness polynomial.
	\end{defi}
	\begin{defi}
		Let	$f$ be a non-constant meromorphic functions in $\mathcal{M}(\mathbb{C})$ and $c\in\mathbb{C}$. By $Z^{-}(f-c)$ we denote the zero set of $f-c$ in the right half plane, i.e., $Z^{-}(f-c)=\{s : f-c=0 ;\;Re(s)<0\;\}.$  \par Similarly  we define, $Z^{+}(f-c)=\{s : f-c=0 ;\;Re(s)>0\;\}.$
	\end{defi}

 Early uniqueness results for $L$-functions were primarily concerned with the sharing of isolated values.
 Such results demonstrated that under suitable hypotheses, the coincidence of a single value or a small number of values could force two $L$-functions to be identical.
 However, these conclusions often relied on restrictive assumptions, such as excluding certain exceptional values or imposing strong structural constraints.
 
 Subsequent developments extended this perspective from isolated values to finite sets and polynomial images.
 Li \cite{Li_Adv_11} initiated this transition by considering the sharing of two values in the extended Selberg class under the assumption of a common functional equation with $a(1)=1$.
 Later, Li--Du--Yi \cite{Li_Du_Yi_Comp.var(22)} strengthened this approach by allowing larger sets of shared values. However, for an arbitrary set  Li--Du--Yi \cite{Li_Du_Yi_Comp.var(22)} obtained the following theorem:
 \begin{theoA}\cite{Li_Du_Yi_Comp.var(22)}
 	Let $P(w)$ be a monic polynomial of degree $q$ without multiple zeros such that
 	\[
 	P'(w)=q(w-e_1)^{q_1}(w-e_2)^{q_2},
 	\]
 	where $e_1$ and $e_2$ are two distinct finite complex numbers, and let
 	$S=\{w\in\mathbb{C}: P(w)=0\}$. Suppose that $\mathcal{L}_1$ and $\mathcal{L}_2$
 	are two non-constant $L$-functions in the extended Selberg class $\mathcal{S}^{\#}$
 	with $a(1)=1$, satisfying the same functional equation and
 	$E_{\mathcal{L}_1}(S)=E_{\mathcal{L}_2}(S)$. If $q\geq 4$ and
 	$P(e_1)=\pm P(e_2)$, then $\mathcal{L}_1=\mathcal{L}_2$.
 \end{theoA}

 Building on this, the present authors \cite{Kundu_Banerjee_Comp.var} eliminated the restriction on the cardinality of the shared set in the case of
 $L$-functions of positive degree.
 \begin{theoB}\cite{Kundu_Banerjee_Comp.var}
 	If two non-constant $L$-functions $\mathcal{L}_1$ and $\mathcal{L}_2$ in the
 	extended Selberg class $\mathcal{S}^{\#}$ have positive degrees, satisfy the
 	same functional equation with $a(1)=1$ and share a set
 	$S=\{\alpha_1,\alpha_2,\ldots,\alpha_t\}$ CM, where $t\geq 1$ is a positive
 	integer and $\alpha_1,\alpha_2,\ldots,\alpha_t$ are $t$ distinct complex
 	constants, then $\mathcal{L}_1=\mathcal{L}_2$.
 \end{theoB}
 
  Further advances were made by Khoai and An \cite{Khoai_An_Ramanujan.J}, who established criteria under which the zero set of a polynomial serves as a unique  range set for $L$-functions and identified conditions ensuring that the polynomial itself is a strong uniqueness polynomial for $L$-functions. Their
  results are stated below.
  \begin{theoC}\cite{Khoai_An_Ramanujan.J}
  	Let $P$ be a polynomial satisfying $P(1)P'(1)\neq 0$. Then $P$ is a strong
  	uniqueness polynomial for $L$-functions of positive degree with $a(1)=1$.
  \end{theoC}
   
 \begin{theoD}\cite{Khoai_An_Ramanujan.J}
 	Let $P$ be a uniqueness polynomial for $L$-functions. Suppose	that P has no multiple zeros and $P(1)\not= 0$. Then the zero set of $P$ is a unique range set for $L$-functions having positive degrees, with $a(1)=1$, counting multiplicities.
 \end{theoD}
 
Combining {\it{Theorems C and D}} yields the following consequence. 
\begin{theoE}\cite{Khoai_An_Ramanujan.J}
	If two non-constant $L$-functions $\mathcal{L}_1$ and $\mathcal{L}_2$ in the
	extended Selberg class $\mathcal{S}^{\#}$ have positive degrees and share the
	zero set of a polynomial $P$, that is,
	\[
	S=\{w\in\mathbb{C}: P(w)=0\}
	\]
	counting multiplicities, where $P(1)\cdot P'(1)\neq 0$, then
	$\mathcal{L}_1=\mathcal{L}_2$.
\end{theoE}
 
 These works reveal a clear conceptual progression: from the sharing of single values, to finite sets and eventually to polynomial-generated sets.
 Along this path, assumptions on functional equations, degree conditions and polynomial structures were gradually weakened.
 	This progression naturally raises the following questions.
 
 \begin{ques}
 	Can the assumption that $\mathcal{L}_1$ and $\mathcal{L}_2$ satisfy the same
 	functional equation be removed?
 \end{ques}
 
 \begin{ques}
 	Is it possible to further relax the hypotheses in Theorems A and C?
 \end{ques} 
 \par
 The foregoing questions motivate us to study the uniqueness problem in a more general and unified setting. 
 In particular, we show that two distinct $L$-functions belonging to the extended Selberg class $\mathcal{S}^{\#}$ cannot share any finite subset of complex numbers counting multiplicities. 
 Equivalently, if two such functions share a finite set CM, then they must coincide identically. 
 Thus, no finite subset of $\mathbb{C}$ can serve as a CM-shared set for two distinct $L$-functions in $\mathcal{S}^{\#}$. 
 Our first main result is stated below.
 
 \begin{theo}\label{t1.1}
 	If two non-constant $L$-functions $\mathcal{L}_1$ and $\mathcal{L}_2$ in the extended Selberg class $\mathcal{S}^{\#}$ have positive degree with $a(1)=1$ and share a set
 	\[
 	S=\{\alpha_1,\alpha_2,\ldots,\alpha_t\}
 	\]
 	counting multiplicities, where $t\ge 1$, then $\mathcal{L}_1=\mathcal{L}_2$.
 \end{theo}
 
 \begin{rem}
 	The absence of any assumption requiring $\mathcal{L}_1$ and $\mathcal{L}_2$ to satisfy a common functional equation demonstrates that Theorem~\ref{t1.1} constitutes a genuine and substantial improvement over earlier uniqueness results.
 	
 	Moreover, suppose there exists a polynomial $P$ such that
 	\[
 	P(\mathcal{L}_1)=c\,P(\mathcal{L}_2)
 	\]
 	for some constant $c\in\mathbb{C}$. 
 	This immediately implies that $\mathcal{L}_1$ and $\mathcal{L}_2$ share the zero set
 	\[
 	\{z\in\mathbb{C}: P(z)=0\}
 	\]
 	counting multiplicities. 
 	An application of Theorem~\ref{t1.1} therefore yields $\mathcal{L}_1=\mathcal{L}_2$.
 	
 	Consequently, the additional hypotheses imposed in Theorems~C, D and E become redundant. 
 	Similarly, the structural restrictions on the polynomial $P$ assumed in Theorem~A can be completely removed. 
 	Hence, Theorem~\ref{t1.1} not only generalizes but also strictly strengthens Theorems~A, C, D and E.
 \end{rem}
 
In Theorem~\ref{t1.1}, the requirement that the $L$-functions $\mathcal{L}_i$ $(i=1,2)$ have positive degree is essential to the validity of the result and plays a decisive role in the proof. This naturally leads to the question of whether the conclusion of Theorem~\ref{t1.1} remains valid in the setting of degree-zero $L$-functions.

To investigate this situation, we turn our attention to $L$-functions of degree zero belonging to the extended Selberg class $\mathcal{S}^{\#}$. Since degree-zero functions exhibit substantially different structural behavior, we impose the additional assumption that the functions satisfy the same functional equation. Within this framework, we are able to establish a complementary uniqueness result, which addresses the degree-zero case and completes the picture suggested by Theorem~\ref{t1.1}.

 \begin{theo}\label{t1.2}
 	If two non-constant $L$-functions $\mathcal{L}_1$ and $\mathcal{L}_2$ in the extended Selberg class $\mathcal{S}^{\#}$ have degree zero, satisfy the same functional equation with $a(1)=1$ and share a set
 	\[
 	S=\{\alpha_1,\alpha_2,\ldots,\alpha_n\}
 	\]
 	counting multiplicities, where $n\ge 1$, then $\mathcal{L}_1=\mathcal{L}_2$.
 \end{theo}
 
 Taken together, Theorems~\ref{t1.1} and \ref{t1.2} provide a unified perspective on the uniqueness of $L$-functions in the extended Selberg class, covering both positive- and zero-degree cases.
 A systematic comparison of these results with earlier contributions, highlighting both the evolution and the improvements achieved, is presented in the table below.
 \begin{table}[htbp]
 	\centering
 	\renewcommand{\arraystretch}{1.2}
 	\caption{Evolution of uniqueness results for $L$-functions}
 	\small
 	\begin{tabular}{|| 
 			>{\centering\arraybackslash}p{2.2cm}||
 			>{\centering\arraybackslash}p{1.3cm}||
 			>{\centering\arraybackslash}p{1.3cm}||
 			>{\centering\arraybackslash}p{2.8cm}||
 			>{\centering\arraybackslash}p{1.0cm}||
 			>{\centering\arraybackslash}p{3.2cm} ||}
 		\hline\hline
 		\textbf{Reference} &
 		\textbf{Class} &
 		\textbf{Degree} &
 		\textbf{Shared data} &
 		\textbf{Mult.} &
 		\textbf{Additional assumptions} \\
 		\hline\hline
 		
 		Steuding \cite{Steuding_Sprin(07)} &
 		$\mathcal{S}$ &
 		$>0$ &
 		Single value $c$ &
 		CM &
 		$c\neq 1$ \\
 		\hline
 		
 		Li \cite{Li_Adv_11} &
 		$\mathcal{S}^{\#}$ &
 		arbitrary &
 		Two values ${c_1,c_2}$ &
 		IM &
 		Same functional equation \\
 		\hline
 		
 		Li--Du--Yi \cite{Li_Du_Yi_Comp.var(22)} &
 		$\mathcal{S}^{\#}$ &
 		$>0$ &
 	$	\{c_1,c_2,c_3\}$ &
 		CM &
 		Same functional equation \\
 		\hline
 		
 		Li--Du--Yi \cite{Li_Du_Yi_Comp.var(22)}\newline (Theorem A) &
 		$\mathcal{S}^{\#}$ &
 		arbitrary &
 		Zero set of polynomial $P$ &
 		CM &
 		Structure on $P$, same functional equation \\
 		\hline\hline
 		
 		Kundu--Banerjee \cite{Kundu_Banerjee_Comp.var} (Theorem B) &
 		$\mathcal{S}^{\#}$ &
 		$>0$ &
 		Finite set $\{\alpha_1,\dots,\alpha_t\},\;t\geq1$ &
 		CM &
 		Same functional equation \\
 		\hline
 		
 		Khoai--An \cite{Khoai_An_Ramanujan.J}\newline (Theorem C) &
 		$\mathcal{S}^{\#}$ &
 		$>0$ &
 		Polynomial $P(\mathcal{L})$ &
 		-- &
 		$P(1)P'(1)\neq 0$ \\
 		\hline
 		
 		Khoai--An \cite{Khoai_An_Ramanujan.J}\newline (Theorem D) &
 		$\mathcal{S}^{\#}$ &
 		$>0$ &
 		Zero set of $P$ &
 		CM &
 		$P$ uniqueness polynomial, $P(1)\neq 0$ \\
 		\hline\hline
 		
 		Present work (Theorem~\ref{t1.1}) &
 		$\mathcal{S}^{\#}$ &
 		$>0$ &
 		Arbitrary finite set &
 		CM &
 		No condition on functional equation \\
 		\hline
 		
 		Present work (Theorem~\ref{t1.2}) &
 		$\mathcal{S}^{\#}$ &
 		$=0$ &
 		Arbitrary finite set &
 		CM &
 		Same functional equation \\
 		\hline\hline
 		
 	\end{tabular}
 \end{table}

		\par  Before discussing the next results, we need to discuss certain things. Let $f$ and $g$ be two non-constant meromorphic functions and consider a finite value $a\in\mathbb{C}$. In \cite{B.Q.Li_Proc.am(21)}, Li defined that $f-a$ and $g-a$ have enough common zeros if $f-a$ and $g-a $ have same zeros with same multiplicities except an exceptional set $G$ (say) of their zeros such that $n(r,G)=o(r)$ as $r\lra\infty$. Here by $n(r,G)$ we denote the counting function of $G$, i.e., the number of points in $G\cap \{|{s}|\leq r\}$ counted according it's multiplicity. The order of a set $G$ is denoted by $\rho(G)$ and is defined in a standard way as follows: $$\rho(G)=\limsup_{r\lra\infty}\frac{\log n(r,G)}{\log r},$$
	where $n(r,G)$ is an increasing function. According to (see p.17, \cite{Hayman}) we know that $G$ is said to be of order $k$ convergence type if $\int_{r_0}^{\infty}\frac{n(r,G)}{r^{k+1}}\; dr$ converges.

	\par In \cite{Li_Math.Z.}, Li proved a uniqueness relation between $\mathcal{L}_1$ and $\mathcal{L}_2$ in terms of their zeros. Li obtained the following result.
	\begin{theoF}
		Two non constant $L$-functions $\mathcal{L}_1$ and $\mathcal{L}_2$ satisfying the same functional equation
		are identically equal if $Z^+(\mathcal{L}_1)\subseteq Z^+(\mathcal{L}_2)$ except possibly a set $G \subseteq Z^+(\mathcal{L}_1)$ of order 1
		convergence type.
	\end{theoF}
	\par 	
	Motivated by the result of Li \cite{Li_Math.Z.}, we revisit the findings of that work by considering the sharing of an arbitrary value in $\mathbb{C}$. The second aim of this paper is to investigate similar uniqueness results under comparable conditions, but in the more general context of arbitrary value sharing in $\mathbb{C}$.

	\begin{theo}\label{t1.3}
		If two non-constant $L$-functions $\mathcal{L}_1$, $\mathcal{L}_2$ in the extended Selberg class $\mathcal{S}^{\#}$, satisfy the same functional equation with $a(1)=1$ and share  $\{\alpha\}$ CM, except possibly a set $G$ of order one convergence type. The the following hold: \\{\em(}i{\em)} if $d_{\mathcal{L}_i}\not=0; \;i=1,2$, then $\mathcal{L}_1=\mathcal{L}_2$.\\{\em(}ii{\em)} if $d_{\mathcal{L}_i}=0;\;i=1,2$, and $\alpha\not=1$ then $\mathcal{L}_1=\mathcal{L}_2$.
	\end{theo}
	
	 \begin{theo}\label{t1.4}
	 	If two non-constant $L$-functions $\mathcal{L}_1$, $\mathcal{L}_2$ in the extended Selberg class $\mathcal{S}^{\#}$ have positive degree, with $a(1)=1$, also $Z^{-}(\mathcal{L}_1-c)=Z^{-}(\mathcal{L}_2-c)$ and $Z^{+}(\mathcal{L}_2-c)= Z^{+}(\mathcal{L}_1-c)$ except possibly a set $G$ of order one convergence type, then they are identical.
	 \end{theo}
	 Theorems~\ref{t1.3} and \ref{t1.4} demonstrate that comparable rigidity persists within the extended Selberg class $\mathcal{S}^{\#}$, even under more flexible hypotheses. In particular, Theorem~\ref{t1.3} treats both the positive-degree and degree-zero cases under complete multiplicity sharing of a single value, while Theorem~\ref{t1.4} addresses a complementary sharing of shifted zero sets. The following table summarizes these results alongside Li's \cite{Li_Math.Z.} theorem, highlighting the class of $L$-functions considered, the nature of the shared data, multiplicity conditions and the additional assumptions involved.
	\begin{table}[htbp]
		\centering
		\renewcommand{\arraystretch}{1.25}
		\caption{Comparison of uniqueness theorems via zero and value sharing for $L$-functions}
		\small
		\begin{tabular}{||
				>{\centering\arraybackslash}p{2.5cm}||
				>{\centering\arraybackslash}p{1.2cm}||
				>{\centering\arraybackslash}p{1.2cm}||
				>{\centering\arraybackslash}p{1.5cm}||
				>{\centering\arraybackslash}p{1.2cm}||
				>{\centering\arraybackslash}p{3.9cm}||}
			\hline\hline
			\textbf{Reference} &
			\textbf{Class} &
			\textbf{Degree} &
			\textbf{Shared data} &
			\textbf{Mult.} &
			\textbf{Additional assumptions} \\
			\hline\hline
			
			Li \cite{Li_Math.Z.} &
			$L$-functions &
			arbitrary &
			$Z^{+}(\mathcal{L}_1)\subseteq Z^{+}(\mathcal{L}_2)$ &
			-- &
			Same functional equation; exception set $G$ of order one convergence type \\
			\hline
			
			Present work (Theorem~\ref{t1.3})(i) &
			$\mathcal{S}^{\#}$ &
			$d_{\mathcal{L}_i}>0$ &
			Single value $\{\alpha\}$ &
			CM &
			Same functional equation; exception set $G$ of order one convergence type \\
			\hline
			
			Present work (Theorem~\ref{t1.3})(ii)& 
			$\mathcal{S}^{\#}$ &
			$d_{\mathcal{L}_i}=0$ &
			Single value $\{\alpha\}$ &
			CM &
			Same functional equation, $\alpha\neq1$; exception set $G$ of order one convergence type \\
			\hline
			
			Present work (Theorem~\ref{t1.4}) &
			$\mathcal{S}^{\#}$ &
			$>0$ &
			$Z^{-}(\mathcal{L}-c)$ and $Z^{+}(\mathcal{L}-c)$ &
			-- &
			 complementary sharing of zeros except a set $G$ of order one convergence type \\
			\hline\hline
			
		\end{tabular}
	\end{table}

	\section{Lemmas}
	\begin{lem}\label{l2.1}\em{(see  Theorem 1.14, \cite{C.C Yang-H.X.Yi})}
		Let $f$, $g\in M(\mathbb{C})$ and let   $\rho(f)$ and $\rho(g)$ be  the order of $f$ and $g$  respectively. Then
		$$\rho(f\cdot g)\leq \max\{\rho(f),\rho(g)\}.$$
	\end{lem}

	\begin{lem}\label{l2.3}{\em{(P. 145, \cite{Steuding_Sprin(07)})}}
		Let $\mathcal{L}$ be a non-constant $L$-function with positive degree $d_{\mathcal{L}}$, then the number of zeros (counting multiplicities) of $\mathcal{L}-c$ for some $c\in\mathbb{C}$ in region $Re(s)>0$ and $|Ims|\leq T$ is $$N_\mathcal{L}^{c}(T)=\frac{d_{\mathcal{L}}}{\pi}T\log\frac{T}{e}+\frac{T}{\pi}\log(\lambda Q^2)+O(\log T),$$ where $\lambda=\prod_{j=1}^{K}\lambda_j^{2\lambda_j}$, and $\lambda_{j},Q,K$ are defined as in axiom (iii).
	\end{lem}
	\begin{lem}\label{l2.4}\cite{Goenk-Haan-Ki_Math. Z.(14)}
		Suppose that $\mathcal{L}$ be a non-constant $L$-function in the extended Selberg class  $\mathcal{S}^{\#}$ having positive degree $d_{\mathcal{L}}$, then for any fixed complex number $c\not=0$,
		there exist positive constants $A_1,\; B_1$ and $C_1$ depending at most on $K$ and the $\mu_j$ and $\lambda_j$
		, such that
		\\{\em{(i)}} $N_{\mathcal{L}-c}(-U,-A_1) = N_{\mathcal{L}-c}(-U,-A_1; B_1) =\frac{d_{\mathcal{L}}}{2}U+ O(1)$, as $U\lra+\infty$;
		\\{\em{(ii)}} each zero of $\mathcal{L}-c$ in $\sigma \leq -A_1$ is within $|s_n|^
		{-C_1 \log |s_n|}$of a trivial zero $s_n$ of $\mathcal{L}$;
		\\{\em{(iii)}} all the zeros of $\mathcal{L}-c$ in $\sigma \leq -A_1$ are simple.
		Here $N_{\mathcal{L}-c}(-U,-A_1;B_1)$ denote the zeros of $\mathcal{L}-c$  having real part in $[-U,-A_1]$ and imaginary part in $[-B_1,B_1]$.
	\end{lem}
	\begin{lem}\label{l2.5}
		Let $\mathcal{L}_1$, $\mathcal{L}_2$ be two non-constant $L$-functions with positive degree, and $\frac{\mathcal{L}_1(s)}{\mathcal{L}_2(s)}=R(s)\frac{\ol {\mathcal{L}_1(1-\ol s)}}{\ol {\mathcal{L}_1(1-\ol s)}}$ for some rational function $R(s)$ and $\beta$ be a non-zero constant in $\mathbb{C}$ and $\mathcal{L}_1^t-\beta\mathcal{L}_2^t\not=0$ and $R(s)\lra r\in\mathbb{C}\backslash\{0\}$ as $|s|\lra\infty$. Then there exist a large  $A>0$ such that $\mathcal{L}_1,\;\mathcal{L}_2$, $\mathcal{L}_1^t-\beta\mathcal{L}_2^t$ and $\prod_{j=1}^{K}(\Gamma(\lambda _{2_j}s + \nu_{2_j}))^{-1}$ have the same zeros (irrespective of their multiplicities) in the region $\{s \in\mathbb{C} : Re (s) <- A\; {\text{and}}\;|Im (s)| < B\}$. Here $A, B > 0$ are large constants.
	\end{lem}
	\begin{proof}		\bea\label{e2.1} \mathcal{L}_i(s)=\chi_i(s)\ol{\mathcal{L}_i(1-\ol{s})},\;\;i=1,2;\eea where $\chi_i(s)=\frac{\omega_i Q_i^{1-2s}\prod_{j=1}^{K}\Gamma(\lambda_{i_j}(1-s)+\ol{\nu}_{i_j})}{\prod_{j=1}^{K}\Gamma(\lambda_{i_j}s+\nu_{i_j})}$.\\ In the negative half plane the zeros of $\mathcal{L}_i,i=1,2;$ coming from the poles of  the gamma factors in their functional equations, are denoted as the trivial zeros. Also by the given condition we have $\frac{\chi_1(s)}{\chi_2(s)}=R(s)$ where $s=\sigma+ it$. \\{\bf{Case-1}} If $t=1$ then, \beas\label{e2.2} \mathcal{L}_{1}(s)-\beta\mathcal{L}_{2}(s)&=&\chi_2(s)({R(s)\ol{\mathcal{L}_1(1-\ol{s})}}-\beta{\ol{\mathcal{L}_2(1-\ol{s})}}).\eeas  \\{\bf{Subcase-1}} Now if $R(s)\lra\beta$ as $|s|\lra\infty$, and $R(s)=\beta$ a constant function, then \beas {R(s)\ol{\mathcal{L}_1(1-\ol{s})}}-\beta{\ol{\mathcal{L}_2(1-\ol{s})}}=\frac{a_1(k)-a_2(k)}{k^{1-s}}(1+o(1)),\eeas as $\sigma\lra-\infty$, and $a_1(k)-a_2(k)\not=0$ for some $k$. \\{\bf{Subcase-2}} $R(s)$ is not a constant function and $R(s)\lra\beta$. Clearly it implies there exist some non-zero constant $\alpha$ such that we can write $R(s)=\beta+\alpha s^{-n}(1+o(1))$ as $|s|\lra\infty$. Then we have,
		\beas {R(s)\ol{\mathcal{L}_1(1-\ol{s})}}-\beta{\ol{\mathcal{L}_2(1-\ol{s})}}=\alpha s^{-n}(1+O(\frac{1}{2^{1-\sigma}}))(1+o(1)),\eeas
		as $\sigma\lra-\infty$.
		\\{\bf{Subcase-3}} If $R(s)\lra r\not=\beta$ as $|s|\lra\infty$, then
		
		\beas {R(s)\ol{\mathcal{L}_1(1-\ol{s})}}-\beta{\ol{\mathcal{L}_2(1-\ol{s})}}={(r-\beta)}(1+O(\frac{1}{2^{1-\sigma}})),\eeas as $\sigma\lra-\infty$
		\par Now from {\bf{Subcases-1, 2, 3 }} we can have $A\;\text{and}\;B>0$ such that $R(s)\ol{\mathcal{L}_1(1-\ol{s})}-\beta{\ol{\mathcal{L}_2(1-\ol{s})}}$ has no zeros lie within $\sigma<-A$ and $|Im(s)|<B$.	\;\;	\\{\bf{Case-2}} If $t\geq 2$.
		 \par Now clearly from (\ref{e2.1}) we have, \beas \mathcal{L}_{i}^t=(\chi_i(s))^t(\ol{\mathcal{L}_i(1-\ol{s})})^t,\;\;\;{\text{for}}\;\;i=1,2.\eeas  Now from above we have \bea\label{e2.2} \mathcal{L}_{1}^{t}(s)-\beta\mathcal{L}_{2}^t(s)&=&\prod_{i=1}^{t}\left({\mathcal{L}_1({s})}-c_i{\mathcal{L}_2({s})}\right),\\\nonumber&=&{\chi_2(s)^t}\prod_{i=1}^{t}\left({R(s)\ol {\mathcal{L}_1(1-\ol {s})}}-c_i\ol {{\mathcal{L}_2(1-\ol {s})}}\right)\eea where $c_i,\;1\leq i\leq t$ are distinct roots of the equation $x^t-\beta$.
		 \par Now here as we proceed in {\bf{Case-1}}, by doing similarly we can have a $A_i>0$ such that ${R(s)\ol {\mathcal{L}_1(1-\ol {s})}}-c_i\ol {{\mathcal{L}_2(1-\ol {s})}}$ has no-zero in $\sigma<-A_i$. Now considering $A=\max_{1\leq j\leq t}\{A_i\}$, $\prod_{i=1}^{t}\left({R(s)\ol {\mathcal{L}_1(1-\ol {s})}}-c_i\ol {{\mathcal{L}_2(1-\ol {s})}}\right)$ has no-zero with in $\sigma<-A$ and $Im(s)<B$ for some $B>0$.
		 Hence $\mathcal{L}_{1}^{t}(s)-\beta\mathcal{L}_{2}^t(s)$ and $\chi_2(s)
		 $ has same zero in $\sigma<-A$. \par Therefore in $\{s:Re(s)<-A,\;|Im(s)|<B\}$; $\mathcal{L}_1(s)$, $\mathcal{L}_2(s)$ and $\mathcal{L}_1^t(s)-\beta\mathcal{L}_2^t(s)$ and $\chi_2(s)$ have same zeros (irrespective of their multiplicities) which are actually poles of $\prod_{j=1}^{K}\Gamma(\lambda _{2_j}s + \nu_{2_j})$. And from {\bf{Case-1}} and {\bf{Case-2}} the results follows immediately.
		 \end{proof}

	\begin{lem}{\label{l2.6}}\cite{Li_Math.Z.}
		Let \( H \) be an entire function satisfying \( |H(s)| = O(e^{k|s|}) \) for some real number \( k \) and for a fixed \( t = t_1 \),
		\[
		|H(\sigma + i t_1)| = O\left( \frac{1}{e^{\eta |\sigma|}} \right)
		\]
		for some \( \eta > 0 \) as \( \sigma \to \pm\infty \). Then \( H \equiv 0 \).
	\end{lem}
	\begin{lem}{\label{l2.7}}\cite{Hayman}If  nonzero complex numbers $b_n$, such that \( \sum |b_n|^{-1} \) converges, then the product
		\[
		E(s) = \prod_{n=1}^{\infty} \left(1 - \frac{s}{b_n}\right)
		\]
		is an entire function, whose zero set is \( D \), and satisfies the estimate
		\[
		\log |E(s)| \leq \int_0^{|s|} \frac{n(r, D)}{r} dr + |s| \int_{|s|}^{\infty} \frac{n(r, D)}{r^2} dr.
		\]
	\end{lem}
	
	\section{proofs of theorems}
	\begin{proof}[Proof of Theorem \ref{t1.1}]
		First let us assume $\mathcal{L}_1\not\equiv\mathcal{L}_2$. It is given that $\mathcal{L}_1$ and $\mathcal{L}_2$ share the set $S=\{\alpha_1,\alpha_2,\ldots,\alpha_t\}$ CM and has only pole at $s=1$, then we can write it as \bea\label{e3.1} \frac{(\mathcal{L}_1-\alpha_1)(\mathcal{L}_1-\alpha_2)\ldots(\mathcal{L}_1-\alpha_t) }{(\mathcal{L}_2-\alpha_1)(\mathcal{L}_2-\alpha_2)\ldots(\mathcal{L}_2-\alpha_t) }=(s-1)^ke^{p(s)},\eea for some integer $k$. Since $L$-function with positive degree has order one [see p. 150, \cite{Steuding_Sprin(07)}], then from {\it{Lemma \ref{l2.1}}}, we have $p(s)$ a polynomial of degree one. Let us consider $p(s)=\hat{a}s+\hat{b},$ where $\hat{a},\hat{b}$ are some complex constants. Next we will show $(s-1)^ke^{\hat{a}s+\hat{b}}=1$. If $0,1\not\in S$ then taking $R(s)=\sigma\lra\infty$ it can be shown easily. Here in this theorem we will consider only the case when $0,1\in S$.
		
		\par Here first we will show $\mathcal{L}_1$ and $\mathcal{L}_2$ share all the trivial zeros with multiplicities in negative half plane, except finitely many. Assume $\{s_n\}_{n=1}^\infty$ and $\{t_n\}_{n=1}^\infty$ are two distinct sequence of zeros $\mathcal{L}_1$ and $\mathcal{L}_2$. If there is a multiple zero of multiplicity $m$ then it appears $m$ times in the sequence. Since $s_n$ and $t_n$ are distinct so it is possible to get a positive $\delta=min{|s_n-t_n|}$. From {\it{Lemma \ref{l2.4}}} and {\it{Subcase-1.2}} of {\it{Theorem 1.2}} in \cite{Kundu_Banerjee_Comp.var} we can have the sequence $\{r_n\}$ of zeros of $\mathcal{L}_2-\alpha_1$ and $\mathcal{L}_1-\alpha_i$, and each $r_n$ lies within $|s_n|^{-C_1\log |s_n|}$ of $s_n$. Again $r_n$ is also zero of $\mathcal{L}_2-\alpha_1$, hence it also lies 
		 with in $|t_n|^{-C_1\log |t_n|}$ of $t_n$. 
		Therefore \beas |s_n-t_n|\leq |s_n-r_n|+|t_n-r_n|\lra 0\;\text{as}\;n\lra\infty.\eeas 
		\\Again for some large $N$, such that  \beas \delta\leq |s_n-t_n|<\frac{\delta}{2},\;\text{for}\; n>N,\eeas a contradiction. 
		\par Now if $\{r_n\}$ is zero of $\mathcal{L}_1$ which is distinct from $\{t_n\}$, in the negative half plane, therefore $r_n$ coincides with the trivial zeros $s_n$ of $\mathcal{L}_1$, for some $n>n_1$. Since $r_n$ is also zero $\mathcal{L}_2-\alpha_1$, then from {\it{Lemma \ref{l2.4}}}, we know each $r_n$ lies within $|t_n|^{-C_1\log |t_n|}$ of trivial zero $t_n$ of $\mathcal{L}_2$. Hence for some $n$, $s_n$ lie within $|t_n|^{-C_1\log |t_n|}$ of trivial zero $t_n$.  Therefore we can have a large number $K$, such that  \beas \delta\leq |s_n-t_n|<\frac{\delta}{2},\;\text{for}\; n>K,\eeas a contradiction.  
		
		\par  Hence $s_n=t_n$ for some $n>N_0$, and therefore we can have a rational function $R$, so that $\frac{\chi_1(s)}{\chi_2(s)}=R(s)Q^{1-2s}$ where $Q=Q_1/Q_2\in \mathbb{Q}$, where $\chi_i(s)=\frac{\omega_i Q_i^{1-2s}\prod_{j=1}^{K}\Gamma(\lambda_{i_j}(1-s)+\ol{\nu}_{i_j})}{\prod_{j=1}^{K}\Gamma(\lambda_{i_j}s+\nu_{i_j})}$ for $i=1,2$. 
		  
		  	\par Now as discussed in Subcase-1.2 of {\it{Theorem 1.2}} of {\cite{Kundu_Banerjee_Comp.var}} here we have a $\{r_n\}_{n=1}^{\infty}$ be a sequence of zeros of $\mathcal{L}_2-\alpha_1$ in some $\sigma<-A\;(A>0)$, $Re(r_{n})\lra-\infty$ as $n\lra \infty$. Now due to the set sharing property, $\{r_n\}_{n=1}^{\infty}$ can be zero of any $\mathcal{L}_1-\alpha_i\;(\not=0)$, as in negative plane $\mathcal{L}_1$, $\mathcal{L}_2$ share trivial zeros. Here assume $\{r_n\}_{n=1}^{\infty}$ be a zero of $\mathcal{L}_1-\alpha_i\;(\alpha_i\not=0)$.
		  Then from \bea\label{e3.7}\frac{\mathcal{L}_1(s)}{\mathcal{L}_2(s)}=R(s)Q^{1-2s}\frac{\ol {\mathcal{L}_1(1-\ol s)}}{\ol {\mathcal{L}_1(1-\ol s)}},\eea putting $s=r_n$ and taking limit $n\lra\infty$ we have from above, $\lim_{n\lra\infty}R(r_n)Q^{1-2r_n}=\text{non-zero constant}$, which implies $Q=1$.
		  \par Therefore we have,
		  \beas \mathcal{L}_1(s)-\mathcal{L}_2(s)=\chi_2(s)\left(R(s)\ol{\mathcal{L}_1(1-\ol{s})}-\ol{\mathcal{L}_2(1-\ol{s})}\right)=\chi_2(s)\left(R(s)\ol{\mathcal{L}_1(1-\ol{s})-\mathcal{L}_2(1-\ol{s})}\right).\eeas 
\;\;		\\{\bf{Case-1.}} $1\in S$ but $0\not\in S$. Without loss of generality let us assume $\alpha_1=1$.
		\par  Again $\mathcal{L}_i$ can be represented by Dirichlet series, i.e., $\mathcal{L}_i(s)=\sum_{n=1}^{\infty}\frac{a_i(n)}{n^s},\;i=1,2$; absolutely convergent for $\sigma>1$ where $a_i(1)=1$.
		\par Now we clearly have, 
		\bea\label{e3.2}\frac{A_1}{n_a^\sigma}\leq |\mathcal{L}_1-1}|\leq\frac{A_2}{n_a^\sigma},\;\;\;{\text{as}\;\;\;\sigma\lra+\infty,\eea where $A_i\;(i=1,2)$ are two non-zero constants and $n_{{a}}=\min\{n\;(\geq 2):a_1(n)\not=0\}$. 
		\par Similarly we can have some constants $B_i,\;i=1,2$ such that, 
		\bea\label{e3.3}\frac{B_1}{n_b^\sigma}\leq |\mathcal{L}_2-1}|\leq\frac{B_2}{n_b^\sigma},\;\;\;{\text{as}\;\;\;\sigma\lra+\infty,\eea where $n_b$ is the smallest integer such that $a_2(n)\not=0$. Therefore clearly we have from (\ref{e3.2}) and (\ref{e3.3}), we have 
		
		 hence we can get from above, \bea\label{e3.5}\lim_{\sigma\lra+\infty}\frac{\mathcal{L}_1-1}{\mathcal{L}_2-1}\cdot\left(\frac{n_a}{n_b}\right)^s=C,\eea for some non-zero constant $C$.\par Now let us consider the following function \bea\label{e3.6}G&=&\frac{\mathcal{L}_1-1}{\mathcal{L}_2-1}\cdot\left(\frac{n_a}{n_b}\right)^s\cdot\frac{(\mathcal{L}_1-\alpha_2)\ldots(\mathcal{L}_1-\alpha_t)}{(\mathcal{L}_2-\alpha_2)\ldots(\mathcal{L}_2-\alpha_t)}\\\nonumber&=&q^s\frac{(\mathcal{L}_1-1)(\mathcal{L}_1-\alpha_2)\ldots(\mathcal{L}_1-\alpha_t)}{(\mathcal{L}_2-1)(\mathcal{L}_2-\alpha_2)\ldots(\mathcal{L}_2-\alpha_t)}=q^s(s-1)^ke^{\hat{a}s+\hat{b}},\eea for some $q=\frac{n_a}{n_b}(\in\mathbb{Q}^{+})$. 
		We can write $q=e^{q'}$, then we can write it as $q^{s}(s-1)^{k}e^{\hat{a}s+\hat{b}}=(s-1)^ke^{(q'+\hat{a})s+\hat{b}}=(s-1)^ke^{a's+\hat{b}}$ where $a'=q'+\hat{a}$.
		Let us consider $a'=a_1+ia_2$ and $\hat{b}=b_1+ib_2$.\par Again taking limit $\sigma\lra\infty$, we have \beas\lim_{\sigma\lra+\infty}\left|q^s\frac{(\mathcal{L}_1-1)}{(\mathcal{L}_2-1)}\cdot\frac{(\mathcal{L}_1-\alpha_2)\ldots(\mathcal{L}_1-\alpha_{t})}{(\mathcal{L}_2-\alpha_2)\ldots(\mathcal{L}_2-\alpha_{t})}\right|=|C|
		&=&\lim_{\sigma\lra+\infty}|(s-1)^ke^{a's+\hat{b}}|\\=|C|&=&\lim_{\sigma\lra+\infty}|\sigma-1+it|^ke^{a_1\sigma-a_2t+b_1}
		.\eeas \par Therefore we must have $a_1=0=k$.
		\par  Now, \beas \lim_{\sigma\lra+\infty}|e^{a's+\hat{b}}|=|C|=\lim_{\sigma\lra+\infty}e^{b_1-a_2t},\eeas $\forall t\in\mathbb{R}$, hence we must have $a_2=0$. Therefore we have $a'=\hat{a}+q'=0$. 
		\par Therefore we have $G=e^{\hat{b}}$ and hence we get from (\ref{e3.6}) \bea\label{e3.7}\frac{(\mathcal{L}_1-1)(\mathcal{L}_1-\alpha_2)\ldots(\mathcal{L}_1-\alpha_{t})}{(\mathcal{L}_2-1)(\mathcal{L}_2-\alpha_2)\ldots(\mathcal{L}_2-\alpha_{t})}=e^{\hat{b}}q^{-s}.\eea

		\par	
		 From the proof of {\it{Lemma }}(a) (see p. 2487, \cite{H.Ki_Adv.Math(12)}) we know for some sufficiently large constant $\kappa(>0)$ and large $\kappa_0(>0)$, $\mathcal{L}_1$, $\mathcal{L}_2$, $\mathcal{L}_1-\mathcal{L}_2$ and the Gamma function $\prod_{j=1}^{K}(\Gamma(\lambda _{2_j}s + \nu_{2_j}))^{-1}$, in their functional equation have same zeros 
		in the region $Re(s)<-\kappa_0$ and $|Im(s)|<\kappa$ and  $\mathcal{L}_1$, $\mathcal{L}_2$ and $\mathcal{L}_1-\mathcal{L}_2$ have no zero in $Re(s)\geq \kappa_0$.
		\par Again Gamma function $(\Gamma(s))^{-1}$ has zeros at $s=-1,-2,\ldots,-n,\ldots$, therefore   
		$\prod_{j=1}^{K}(\Gamma(\lambda _{2_j}s + \nu_{2_j}))^{-1}$ has zeros at $s=\frac{-n-\nu_{2_j}}{\lambda_{2_j}}$ for $n=1,2,\ldots$ and $j=1,2,\ldots,K$. 
		Since it is assumed that $d_\mathcal{L}>0$, we must have $K>0$.
		Here for some $j$ or fixing some $j$ it is possible to get a infinite sequence $\{s_n\}_{n=1}^{\infty}$ of common zeros of $\mathcal{L}_1$, $\mathcal{L}_2$ and $\prod_{j=1}^{K}(\Gamma(\lambda _{2_j}s + \nu_{2_j}))^{-1}$  in $Re(s)<-\kappa_0$ and $|Im(s)|<\kappa$, where for $n\lra\infty$ we have, $Re(s_n)\lra-\infty$.
		
		\par From (\ref{e3.7}) we have, \beas \frac{(\mathcal{L}_1(s_n)-1)(\mathcal{L}_1(s_n)-\alpha_2)\ldots(\mathcal{L}_1(s_n)-\alpha_{t})}{(\mathcal{L}_2(s_n)-1)(\mathcal{L}_2(s_n)-\alpha_2)\ldots(\mathcal{L}_2(s_n)-\alpha_{t})}=e^{\hat{b}}q^{-s_n}\implies 1=|e^{\hat{b}}q^{-s_n}|=e^{Re(\hat{b})}q^{-Re(s_n)},\eeas now $e^{Re(\hat{b})}q^{-Re(s_n)}\lra\infty\;\text{or}\;0$ as $n\lra\infty$ according as $q>1$ or $<1$ but $\{e^{Re(\hat{b})}q^{Re(-s_n)}\}_{n=1}^{\infty}$ is a constant sequence. Hence we must have $q=1$ and then $e^{\hat{b}}=1$.
		\\{\bf{Case-2.}} Let us assume $1,0\in S$. Without loss of generality let us assume $1=\alpha_1,\;0=\alpha_2$.
		\par Then we can write (\ref{e3.7}) as \bea\label{e3.8} \frac{\mathcal{L}_1(\mathcal{L}_1-1)(\mathcal{L}_1-\alpha_3)\ldots(\mathcal{L}_1-\alpha_t)}{\mathcal{L}_2(\mathcal{L}_2-1)(\mathcal{L}_2-\alpha_3)\ldots(\mathcal{L}_2-\alpha_t)}=e^{\hat{b}}q^{-s}.\eea
		\par Since	$\mathcal{L}_1$, $\mathcal{L}_2$  have same trivial zero in negative half plane, so we
		have, $ \frac{\mathcal{L}_1(s)}{\mathcal{L}_2(s)}=R(s)Q^{1-2s}\frac{\ol {\mathcal{L}_1(1-\ol s)}}{\ol{\mathcal{L}_2(1-\ol{s})}}$. Also it is shown that $Q=1$, and hence $ \frac{\mathcal{L}_1(s)}{\mathcal{L}_2(s)}=R(s)\frac{\ol {\mathcal{L}_1(1-\ol s)}}{\ol{\mathcal{L}_2(1-\ol{s})}}$, where $\lim_{|s|\lra\infty}R(s)=\text{non zero constant}$. Using this we have from above(\ref{e3.8}),  \bea\label{e3.9} R(s)\frac{\ol {\mathcal{L}_1(1-\ol s)}(\mathcal{L}_1-1)(\mathcal{L}_1(s)-\alpha_3)\ldots(\mathcal{L}_1(s)-\alpha_t)}{\ol {\mathcal{L}_2(1-\ol s)}(\mathcal{L}_2-1)(\mathcal{L}_2(s)-\alpha_3)\ldots(\mathcal{L}_2(s)-\alpha_t)}=e^{\hat{b}}q^{-s}.\eea
		Putting $s=s_n$ the trivial zero of $\mathcal{L}_1$, $\mathcal{L}_2$ we have from (\ref{e3.9}), \beas \left|\frac{\ol{\mathcal{L}_1(1-\ol{s_n})}}{\ol{\mathcal{L}_2(1-\ol{s_n})}}\right|=\left|e^{\hat{b}}\frac{q^{-s_n}}{R(s_n)}\right|.\eeas \par Now taking limit as $n\lra\infty$ we have from above, $\lim_{n\lra\infty}|e^{\hat{b}}q^{-s_n}/R(s_n)|=1$, 
		 a contradiction. Hence $q$ must be $1$.
			\par Therefore from ({\ref{e3.9}}) we have  \beas R(s)\frac{\ol {\mathcal{L}_1(1-\ol s)}(\mathcal{L}_1-1)(\mathcal{L}_1(s)-\alpha_3)\ldots(\mathcal{L}_1(s)-\alpha_t)}{\ol {\mathcal{L}_2(1-\ol s)}(\mathcal{L}_2-1)(\mathcal{L}_2(s)-\alpha_3)\ldots(\mathcal{L}_2(s)-\alpha_t)}=e^{\hat{b}},\eeas hence $\lim_{n\lra\infty}R(s_n)=e^{\hat{b}}(=c,\;\text{say}).$
	

		\[
		\]
		

		\par Now from (\ref{e3.8}) we have \bea\label{e3.10} (\mathcal{L}_1-\alpha_1)(\mathcal{L}_1-\alpha_2)\ldots(\mathcal{L}_1-\alpha_{t})=c(\mathcal{L}_2-\alpha_1)(\mathcal{L}_2-\alpha_2)\ldots(\mathcal{L}_2-\alpha_{t}).\eea
		\par	Next we will show $\mathcal{L}_1=\mathcal{L}_2$.
		Let us assume, \beas(\mathcal{L}_i-\alpha_1)(\mathcal{L}_i-\alpha_2)\ldots(\mathcal{L}_i-\alpha_{t})=\mathcal{L}_i^t+a_1\mathcal{L}_i^{t-1}+a_2\mathcal{L}_i^{t-1}+\ldots+a_{t-1}\mathcal{L}_i+a_t,\eeas for $i=1,2$, where $a_j= (-1)^j \sum_{1\leq i_1<i_2<\ldots<i_j\leq t}\alpha_{i_1}
		\alpha_{i_2}\ldots \alpha_{i_j}\;\;\;j=1,2,\ldots,t$. 
		\par From (\ref{e3.10}) we have \bea\nonumber
		(\mathcal{L}_1^t-c\mathcal{L}_2^t)+a_1(\mathcal{L}_1^{t-1}-c\mathcal{L}_2^{t-1})+a_2(\mathcal{L}_1^{t-2}-c\mathcal{L}_2^{t-2})+\ldots+a_{t-1}(\mathcal{L}_1-c\mathcal{L}_2)&=&0,\eea i.e., \bea\label{e3.11}&&(\mathcal{L}_1^t-c\mathcal{L}_2^t)+a_1(\mathcal{L}_1^{t-1}-c\mathcal{L}_2^{t-1})+a_2(\mathcal{L}_1^{t-2}-c\mathcal{L}_2^{t-2})+\ldots+a_{t-2}(\mathcal{L}_1^{2}-c\mathcal{L}_2^{2})\\\nonumber&=&-a_{t-1}(\mathcal{L}_1-c\mathcal{L}_2).\eea
		\par Since $\frac{\mathcal{L}_1(s)}{\mathcal{L}_2(s)}=R(s)\frac{\ol {\mathcal{L}_1(1-\ol s)}}{\ol {\mathcal{L}_1(1-\ol s)}}$ for some rational function $R(s)$, therefore  $\mathcal{L}_1$, $\mathcal{L}_2$ have the same set of trivial zeros ($\cup_{j=1}^{K}\cup_{n=1}^{\infty}\{s:s=-\frac{n+\mu_{2_j}}{\lambda_{2_j}}\}$) which comes from the zeros of $\chi_2(s)$. Now from {\it{Lemma \ref{l2.5}}}, for some large positive $k$ and $A$ in the region  $Re(s)<-A$ and $|Im(s)|<B$, $\mathcal{L}_1$, $\mathcal{L}_2$, $\mathcal{L}_1^{i}-c\mathcal{L}_2^{i}$ and $\prod_{j=1}^{K}(\Gamma(\lambda _{2_j}s + \nu_{2_j}))^{-1}$ have the same zeros. Also from {\it{Lemma \ref{l2.5}}} in this negative half plane we can have  $s_1$, a trivial zero of $\mathcal{L}_1$ and $\mathcal{L}_2$ of order $p$, which is also a zero of $\mathcal{L}_1^{i}-c\mathcal{L}_2^{i}$ of order $ip,\;1\leq i\leq t $. Then clearly from the  (\ref{e3.11}) we have $s_1$ is the zero of the left side of multiplicity at least $2p$, whereas the right part has multiplicity $p$, hence a contradiction.

		\par Therefore we must have, either $a_{t-1}=0$ or $\mathcal{L}_1=c\mathcal{L}_2\implies\mathcal{L}_1=\mathcal{L}_2$. \\Now assume $\mathcal{L}_1\not=\mathcal{L}_2$ and $a_{t-1}=0$, then we have, \bea\label{e3.12}&&(\mathcal{L}_1^t-c\mathcal{L}_2^t)+a_1(\mathcal{L}_1^{t-1}-c\mathcal{L}_2^{t-1})+a_2(\mathcal{L}_1^{t-2}-c\mathcal{L}_2^{t-2})+\ldots+a_{t-3}(\mathcal{L}_1^3-c\mathcal{L}_2^3)\\\nonumber&=&-a_{t-2}(\mathcal{L}_1^{2}-c\mathcal{L}_2^{2}). \eea  From {\it{Lemma \ref{l2.5}}}, in some negative half plane we can have a trivial zero   $s_2$ of $\mathcal{L}_1$, $\mathcal{L}_{2}$ of order $q$, which is also a zero of $\mathcal{L}_1^{i}-c\mathcal{L}_2^{i}$ of order $iq,\;1\leq i\leq t $. Clearly from the above equation (\ref{e3.12}) we have $s_2$ is the zero of the left side of multiplicity at least $3q$, whereas the right part has multiplicity $2q$, hence a contradiction. Therefore $a_{t-2}=0$ or $\mathcal{L}_1^2=c\mathcal{L}_2^2$. Again $\mathcal{L}_1^2=c\mathcal{L}_2^2\implies\mathcal{L}_1=\mathcal{L}_2$  \par  Proceeding similarly we will have, $a_{t-i}=0$ or $\mathcal{L}_1^{t-i}=c\mathcal{L}_2^{t-i}$ for $i=1,2,\ldots,t$ and finally we will have $\mathcal{L}_{1}^{t}=c\mathcal{L}_2^{t}$ which implies $\mathcal{L}_1=\omega\mathcal{L}_2$ and taking $\sigma\lra\infty$, we  have $\omega=1$, hence $\mathcal{L}_1= \mathcal{L}_2$.

		\end{proof}
	\begin{proof}[Proof of Theorem \ref{t1.3}] The proof of this theorem is mainly based on the idea of a paper of Li \cite{Li_Math.Z.}.
		\par	Here it is given that $\mathcal{L}_1$ and $\mathcal{L}_2$ share $\alpha$ except a set $G$ of order one convergence type, i.e., $\int_{r_0}^{+\infty}\frac{n(t,G)}{t^2}< +\infty$ for some $r_0\geq 0$. Now let $a_i,i=1,2,\ldots$ be  the non-zero elements of the set $G$ repeated according to its multiplicity then we have \beas \sum_{i=1}^{\infty}|a_i|^{-1}=\int_{0}^{\infty}\frac{d(n(t,G\backslash\{0\}))}{t}\leq \int_{0}^{+\infty}\frac{n(t,G\backslash\{0\})}{t^2}<+\infty.\eeas 
		
		\par	Let us consider the following auxiliary function $$H=\frac{\mathcal{L}_1-\alpha}{\mathcal{L}_2-\alpha}.$$ 	\par Clearly zeros and poles of $H$ come from the set $G$ as well as the poles of $L$ and $g$ respectively. Let us consider two sets $G_1$ and $G_2$  such that $G_1\cup G_2=G\backslash\{0\}$ and the elements of $G_1$ are zeros of $H$ and the elements of $G_2$ are poles of $H$(i.e those point where $\mathcal{L}_2-\alpha$ has zero of order $p$ and $\mathcal{L}_1-\alpha$ has zero of order $q$ and $p>q\;(q\geq0)$).   
		\par	Next let us construct the functions $h_i(s)=\prod_{k=1}^{\infty} \left(1-\frac{s}{a^{i}_k}\right)$, where $a^{i}_k\;(i=1,2)$ are points of $G_i$  arranging as $|a_k^i|\leq |a_{k+1}^i|$, repeated according to their multiplicities.. 
		Since $G$ is of order one convergence type then  
		clearly $G_i\;(i=1,2)$ are also order one convergence type, i.e., $\int_{r_0}^{+\infty}\frac{n(t,G_i)}{t^2}dr<+\infty\;\;(i=1,2)$ for some $r_0\geq 0$. 
		\par We have, \beas \sum |a^i_{k}|^{-1}=\int_{0}^{+\infty}\frac{dn(t,G_i)}{t}\leq \lim_{t\lra\infty}\frac{n(t,G_i)}{t}+\int_{0}^{+\infty}\frac{n(t,G_i)}{t^2}dt=\int_{0}^{+\infty}\frac{n(t,G_i)}{t^2}dt<+\infty.\eeas 
		
		\par Using {\it{Lemma \ref{l2.5}}} we have, \bea\label{e3.13} \log|h_i(s)|&\leq& \int_{0}^{|s|}\frac{n(t,G_i)}{t}dt+|s|\int_{|s|}^{\infty}\frac{n(t,G_i)}{t^2}dt\\\nonumber&=&  \int_{0}^{r_0}\frac{n(t,G_i)}{t}dt+\int_{r_0}^{|s|}\frac{n(t,G_i)}{t}dt+|s|\int_{|s|}^{\infty}\frac{n(t,G_i)}{t^2}dt.   \eea
		
		As we have $\int_{r_0}^{\infty}\frac{n(t,G_i)}{t^2}dt<+\infty$ for some $r_0\geq 0$, it follows that for any small positive $\epsilon$, we can have some large $r\geq r_0$, so that $\int_{r}^{\infty}\frac{n(t,G_i)}{t^2}dt<\epsilon$ 
		\par Again as $n(t,G_i)$ is an increasing function so we get 
		
		\beas &&n(r,G_i)\int_{r}^{r_1}\frac{dt}{t^2}\leq\int_{r}^{r_1}\frac{n(t,G_i)}{t^2}dt\\&&\frac{n(r,G_i)}{r}\leq n(r,G_i)\int_{r}^{r_1}\frac{dt}{t^2}\leq\int_{r}^{r_1}\frac{n(t,G_i)}{t^2}dt<\epsilon,\eeas hence for some large $r$ we have $n(r,G_i)\leq r\epsilon$.
		\par From (\ref{e3.13}), we can have $\log |h_i(s)|\leq 3\epsilon |s|$. As $h_i$'s, ($i=1$, $2$) are entire functions, it follows that $T(r,h_i)=m(r,h_i)$. Then, \bea\label{e3.14} T(r,h_i)=m(r,h_i)=\frac{1}{2\pi}\int_{0}^{2\pi}\log ^{+}|h_i(re^{i\theta})|d\theta =\frac{1}{2\pi}\int_{\theta\in \Theta}\log |h_i(re^{i\theta})|d\theta\leq O(r),   \eea
		where	$\Theta=\{\theta:|h_i(re^{i\theta})|>1\}$.

		\par since $h_2(s)$ is an entire function so for a suitable integer $k, l$ we will have $$F=\frac{h_2(s)(\mathcal{L}_1-\mathcal{L}_2)s^l(s-1)^k}{\mathcal{L}_2-\alpha},$$ an entire function. Here clearly we have, $$T(r,F)=O(r\log r),$$ Hence order of $F$ is less equal to one.\\ Now it is clear there are zeros of $F$ which are neither $\alpha$ points nor trivial zeros of of $\mathcal{L}_1$, $\mathcal{L}_2$. Next we will estimate all those zeros of $F$. Now for an $L$-function $L$ and a complex number $a$ we denote by $n(R,\frac{1}{L-a})$ denotes
		the number of zeros (counting multiplicities) of $L-a$ within $|s| < R$. Also by \cite{Steuding_Sprin(07)} we know on the left half-plane
		$\sigma<0$, the zeros of $L-a$ have bounded imaginary
		parts and the number of these zeros having real part in $[-R, 0]$ is 
	\beas n_{-}(R,\frac{1}{L-a})=\frac{d}{2}R+O(1),\eeas
		where $d$ is the degree of $L$-function.
		\par Therefore for sufficiently large $R$ we have \beas n_{-}(R,\frac{1}{\mathcal{L}_1-\mathcal{L}_2})- n_{-}(R,\frac{1}{\mathcal{L}_1-\alpha})=O(1).\eeas \par Again we denote by $n_{+}(R,\frac{1}{L-a})$ denotes
		the number of zeros (counting multiplicities) of $L-a$ within $|s| < R$. By definition $n(R,\frac{1}{L-a})=n_{+}(R,\frac{1}{L-a})+n_{-}(R,\frac{1}{L-a})$.\\  Also by {\it{Lemma\ref{l2.3}}} we have, \beas n_{+}(R,\frac{1}{L-a})=\frac{d}{\pi}R\log\frac{R}{e}+\frac{R}{\pi}\log(\lambda Q^2)+O(\log R),\eeas where $\lambda=\prod_{j=1}^{K}\lambda_j^{2\lambda_j}$, and $\lambda_{j},Q,K$ are defined as in axiom (iii). Now we have \beas &&n_{+}(R,\frac{1}{\mathcal{L}_1-\mathcal{L}_2})- n_{+}(R,\frac{1}{\mathcal{L}_1-\alpha})\\&\leq&\frac{d}{\pi}R\log\frac{R}{e}+\frac{R}{\pi}\log(\lambda Q^2)+O(\log R)-\frac{d}{\pi}\sqrt{R^2-R_o^2}\log\frac{\sqrt{R^2-R_o^2}}{e}\\&&-\frac{\sqrt{R^2-R_o^2}}{\pi}\log(\lambda Q^2)+O(\log \sqrt{R^2-R_o^2})\\&\leq& \frac{d}{\pi}R\log\frac{R}{e}+\frac{R}{\pi}\log(\lambda Q^2)+O(\log R)-\frac{d}{\pi}R\sqrt{1-\frac{R_o^2}{R^2}}\log\left(\frac{R\sqrt{1-\frac{R_o^2}{R^2}}}{e}\right) \\&&-\frac{R\sqrt{1-\frac{R_o^2}{R^2}}}{\pi}\log(\lambda Q^2)+O(\log \left(R\sqrt{1-\frac{R_o^2}{R^2}} \right)) \leq O(\log R).\eeas
		\par It is easy to check that \beas n(r,\frac{1}{F})\leq n(r,\frac{1}{\mathcal{L}_1-\mathcal{L}_2})-n(r,\frac{1}{\mathcal{L}_2-\alpha})+n(r,\frac{1}{h_2})+O(\log r).\eeas
		Therefore \beas n(r,\frac{1}{F})\leq n(r, \frac{1}{h_2})+ O(\log r),\eeas which implies for some $r_o>0$ \bea\label{e3.15}\int_{r_o}^{\infty}\frac{n\left(r,\frac{1}{F}\right)}{r^2}\leq \int_{r_o}^{\infty}\frac{n(r,\frac{1}{h_2})+O(\lg r)}{r^2}\leq +\infty.\eea Therefore the zero set of $F$ which are not trivial zero of $\mathcal{L}_1$ and $\mathcal{L}_1-\alpha$ also order one convergence type. Here clearly we have the order $F$ is at most 1. Therefore by {\it{Lemma \ref{l2.7}}} we can have an entire function $h(s)$ whose zero set is non-zero zeros of $F$, and then by the
		classic Hadamard factorization theorem (see e.g. [1, p. 384]), we have $F(s)= s^lh(s)e^{as+b}$, for some complex numbers $a,b$, proceeding same as done before we will have $\log |h(s)|={O(|s|)}$. Next we will consider the  following cases. 
		\\{\bf{\underline{Case-1}}} Consider $\alpha\not=1$, $d_\mathcal{L}\not=0$.
		\par then $\frac{\mathcal{L}_1-\mathcal{L}_2}{\mathcal{L}_2-\alpha}=O\left(\frac{1}{2^\sigma}\right)$ as $\sigma\lra\infty$. Also let us consider the case when $\sigma\lra-\infty$, \bea\label{e3.16} \frac{\mathcal{L}_1-\mathcal{L}_2}{\mathcal{L}_2-\alpha}=\frac{\ol {\mathcal{L}_1(1-\ol s)}-\ol {\mathcal{L}_2(1-\ol s)}}{\ol {\mathcal{L}_2(1-\ol s)}-\alpha \chi(s)^{-1}}.\eea 
		
		Now $\chi(s)=\frac{\omega Q^{1-2s}\prod_{j=1}^{K}\Gamma(\lambda_{j}(1-s)+\ol{\nu}_{j})}{\prod_{j=1}^{K}\Gamma(\lambda_{j}s+\nu_{j})}$, from the relation  $\Gamma(s)\Gamma(1-s)=\frac{\pi}{sin\pi s}$, we have $\Gamma(\lambda_{j}s+\nu_{j})=\Gamma(1-\lambda_{j}s-\nu_{j})\frac{sin\pi(\lambda_{j}s+\nu_{j})}{\pi}$. \par Here consider a $\eta=\max_{1\leq j\leq k}\{\frac{im(\mu_j)}{\lambda_j}\}+1$. Now for $|Im(s)|\geq\eta$, $\prod_{j=1}^{k}sin\pi(\lambda_{j}s+\nu_{j})\not=0$. In this region fixing $t$, we can actually bound the value of $\prod_{j=1}^{k}|sin\pi(\lambda_{j}s+\nu_{j})|$ by some non-zero positive number.
		
		 From Stirling’s Approximation we know
		for large $s$, the Gamma function satisfies:
		\beas
		\Gamma(s) &\sim& \sqrt{2\pi} \, s^{s - \frac{1}{2}} e^{-s}
		\\&\sim& (2\pi)^{1/2}e^{(s-1/2)\log s-s}\sim e^{s\log s},	\eeas for sufficiently large $\sigma$ with \( |\arg(s)| < \pi \).
		
		\par Now for some $s$ with $t=\eta$ and  $\sigma\lra -\infty$, 
		\beas &&\frac{\Gamma(\lambda_j (1 - s) + \overline{\mu}_j) }{\Gamma(\lambda_j s + \mu_j) } ={\Gamma(\lambda_j (1 - s) + \overline{\mu}_j) }\Gamma(1-\lambda_{j}s-\mu_{j})\frac{sin\pi(\lambda_{j}s+\nu_{j})}{\pi}
		 \\&&{\sim} e^{(\lambda_j (1 - s) + \overline{\mu}_j)\log (\lambda_j (1 - s) + \overline{\mu}_j)+(1-\lambda_{j}s-\mu_{j})\log (1-\lambda_{j}s-\mu_{j})}e^{-O(t)}\sim e^{2\lambda_{j}(1-s)\log (1-s)}
		\eeas
		 Therefore we have when $\sigma\lra -\infty$, \[
		 \left| Q^{1-2s}\frac{ \prod_{j=1}^r \Gamma(\lambda_j (1 - s) + \overline{\mu}_j) }{ \prod_{j=1}^r \Gamma(\lambda_j s + \mu_j) } \right|
		 \sim 
		 e^{ 2 \sum_{j=1}^r \lambda_j |\sigma| \log |\sigma|},
		 \] 
		 and hence here we have $\chi(s)^{-1}\lra0$, as $\sigma\lra-\infty$. \\ Using this from (\ref{e3.16}), we have $\lim_{\sigma\lra-\infty}\frac{\mathcal{L}_1-\mathcal{L}_2}{\mathcal{L}_2-\alpha}=O(\frac{1}{2^{-\sigma}})$ for some fixed $t$. Now we have \beas |F(s)|\leq |s-1|^k|s|^le^{3\epsilon|s|}O\left(\frac{1}{2^{|\sigma|}}\right),\eeas as $|\sigma|\lra\infty$, fixing some $t=t_1$. Considering $\epsilon$ very small we can have a positive $\kappa$, such that $|F(\sigma+it_1)|=O\left(\frac{1}{e^{\kappa|\sigma|}}\right)$ as $|\sigma|\lra\infty$. Hence by {\it{Lemma \ref{l2.6}}} the result follows.
		\\{\bf{\underline{Case-2}}} When, $\alpha=1$, $d_\mathcal{L}\not=0$. \par As done in {\bf{Case-1}} it can be shown that $\lim_{\sigma\lra-\infty}\frac{\mathcal{L}_1-\mathcal{L}_2}{\mathcal{L}_2-1}=O(\frac{1}{2^{-\sigma}})$ for some fixed $t$.
	\par Now consider the case 	
		$\lim_{\sigma\lra+\infty}\frac{\mathcal{L}_1-\mathcal{L}_2}{\mathcal{L}_2-1}=O(\frac{1}{2^{-\sigma}})$.\\ Now here \beas \frac{\mathcal{L}_1-\mathcal{L}_2}{\mathcal{L}_2-1}=\frac{a_{1_1}-a_{2_1}}{a_{1_1}}+ O\left(\frac{2}{3}\right)^{\sigma},\eeas Clearly $\lim_{\sigma\lra+\infty}\frac{\mathcal{L}_1-\mathcal{L}_2}{\mathcal{L}_2-1}=O(1)\;\text{or}\;0\;\text{or}\;{\infty}$, here we have assumed that $\mathcal{L}_i(s)=\sum_{n=1}^{\infty}\frac{a_{i_n}}{n^s}$. Now from {\it{Lemma \ref{l2.4}}}  and the discussion in  in Subcase-1.2 of {\it{Theorem 1.2}} of {\cite{Kundu_Banerjee_Comp.var}} we can have a $\{T_n\}_{n=1}^{\infty}$ be a sequence of zeros of $\mathcal{L}_i-1\;(i=1,2)$ in some $\sigma<-A\;(\text{for some}\;A>0)$, and $Re\{T_n\}\lra -\infty$ as $n\lra\infty$. Consider a sequence in positive half plane with diverging real part, $\{S'_{n}\;(=1-\ol T_n)\}_{n=1}^{\infty}$. Easy to verify $\chi(1-\ol T_n)\not=0,\; 1$.
		\bea\label{e3.17} \frac{\mathcal{L}_1(1-\ol T_n)-\mathcal{L}_2(1-\ol T_n)}{\mathcal{L}_2(1-\ol T_n)-1}=\frac{\ol {\mathcal{L}_1(T_n)}-\ol {\mathcal{L}_2( T_n)}}{\ol {\mathcal{L}_2( T_n)}- \chi(1-\ol T_n)^{-1}}=0.\eea   
		
		Therefore $\lim_{n\lra+\infty}\frac{\mathcal{L}_1(S'_n)-\mathcal{L}_2(S'_n)}{\mathcal{L}_2(S'_n)-1}=0$, 
		 and hence we have $\frac{\mathcal{L}_1-\mathcal{L}_2}{\mathcal{L}_2-1}=O\left(\frac{2}{3}\right)^{\sigma}$ as $\sigma\lra\infty$. \par Now we have \beas |F(s)|\leq |s-1|^k|s|^le^{3\epsilon|s|}O\left(\frac{2}{3}\right)^{|\sigma|},\eeas as $|\sigma|\lra\infty$, fixing some $t=t_1$. Hence choosing $\epsilon$ small, from {\it{Lemma \ref{l2.6}}} we can have $\mathcal{L}_1=\mathcal{L}_2$.
				\\{\bf{\underline{Case-3}}}	$d_\mathcal{L}=0$ and $\alpha\not=1$. Clearly $\frac{\mathcal{L}_1-\mathcal{L}_2}{\mathcal{L}_2-1}=O\left(\frac{1}{2}\right)^{\sigma}$ as $\sigma\lra \infty$. Also, here $Q>1\;(\text{since}\;d_{\mathcal{L}}=0)$ therefore we have \beas\frac{\mathcal{L}_1-\mathcal{L}_2}{\mathcal{L}_2-1}=\frac{\ol {\mathcal{L}_1(1-\ol s)}-\ol {\mathcal{L}_2(1-\ol s)}}{\ol {\mathcal{L}_2(1-\ol s)}-\alpha Q^{2s-1}}=O\left(\frac{1}{2}\right)^{-\sigma}, \eeas as $\sigma\lra-\infty$. 
				\par Here proceeding similarly as in {\bf{Case-1, 2}} we have the result.

	\end{proof}

\begin{proof}[Proof of Theorem \ref{t1.4}]
Here it is given that $Z^{-}(\mathcal{L}_1-c)=Z^{-}(\mathcal{L}_2-c)$ and $Z^{+}(\mathcal{L}_1-c)=Z^{+}(\mathcal{L}_2-c)$, except at a set $G$ of order one convergence type. Then proceeding similarly as in {\it{Theorem {}}}, we can have a $\hat{h}(s)$ such that $H(s)=\frac{\hat{h}(s)s^l(s-1)^k(\mathcal{L}_1-\mathcal{L}_2)}{\mathcal{L}_2-c}$ is an entire function	and here $\log (\hat{h}(s))=O(e^{|s|})$. \par Here first we will show a relation between $\chi_1(s)$ and $\chi_2(s)$ from the functional equation of $\mathcal{L}_1$, $\mathcal{L}_2$. Suppose $\{s_n\}_{n=1}^{\infty}$ and $\{s'_n\}_{n=1}^{\infty}$ be two distinct sequence of trivial zeros of $\mathcal{L}_1$ and $\mathcal{L}_2$ in the negative half plane. Since $s_n$ and $s'_n$ are distinct so it is possible to get a positive $\delta=min{|s_n-s'_n|}$. Here $\delta$ is the distance between the nearest two distinct zeros of $\mathcal{L}_1$ and $\mathcal{L}_2$ along the sequence $\{s_n\}$ and $\{s'_n\}$. \par From {\it{Lemma \ref{l2.4}}}, and as done in \cite{Kundu_Banerjee_Comp.var}, we can have a sequence $\{t_n\}$ of zeros of $\mathcal{L}_1-c$, and each $t_n$ lies within $|s_n|^{-C_1\log |s_n|}$ of $s_n$. Again $t_n$ is also zero of $\mathcal{L}_2-c$, hence it also lies with in $|s'_n|^{-C_1\log |s'_n|}$ of $s'_n$. Therefore \beas |s_n-s'_n|\leq |s_n-t_n|+|t_n-s'_n|\lra 0\;\text{as}\;n\lra\infty.\eeas
	Therefore for some large $N$, such that  \beas \delta\leq |s_n-s'_n|<\frac{\delta}{2},\;\text{for}\; n>N,\eeas a contradiction. Hence $s_n=t_n$ for $n>N$, and therefore we can have a rational function $R$, so that $\frac{\chi_1(s)}{\chi_2(s)}=R(s)Q^{1-2s}$ where $Q=Q_1/Q_2$.\par Now the real part of $t_n\lra-\infty$ as $n\lra\infty$, using this fact from \beas \frac{\mathcal{L}_1(s)}{\mathcal{L}_2(s)}=R(s)Q^{1-2s}\frac{\ol {\mathcal{L}_1(1-\ol s)}}{\ol {\mathcal{L}_2(1-\ol s)}},\eeas we have $\lim_{n\lra+\infty}R(t_n)Q^{1-2t_n}=1$, implies $Q=1$ and hence we have $\lim_{|s|\lra\infty}R(s)=1$. 
	\par Here we will only consider the case when $c=1$.  \bea\label{e3.18} \frac{\mathcal{L}_1-\mathcal{L}_2}{\mathcal{L}_2-1}=\frac{R(s)\ol {\mathcal{L}_1(1-\ol s)}-\ol {\mathcal{L}_2(1-\ol s)}}{\ol {\mathcal{L}_2(1-\ol s)}- \chi_2(s)^{-1}},\eea As we proceed in {\bf{Case-1}}, for a fixed $t$ from (\ref{e3.18}), $\lim_{\sigma\lra-\infty}\left|\frac{\mathcal{L}_1-\mathcal{L}_2}{\mathcal{L}_2-1}\right|=0$, $\frac{\mathcal{L}_1-\mathcal{L}_2}{\mathcal{L}_2-1}=O\left(\frac{1}{2}\right)^{|\sigma|} $ as $\sigma\lra-\infty$.
	
	\par Now when consider the case $\lim_{\sigma\lra+\infty}\frac{\mathcal{L}_1-\mathcal{L}_2}{\mathcal{L}_2-1}$, clearly the limit can be $0$ or $O(1)$ or diverge. Let's evaluate the limit on the sequence $\{1-\ol t_n\}_{n=1}^{\infty}$.
			\bea\label{e3.19}\lim_{n\lra\infty} \frac{\mathcal{L}_1(1-\ol t_n)-\mathcal{L}_2(1-\ol t_n)}{\mathcal{L}_2(1-\ol t_n)-1}=\lim_{n\lra\infty}\frac{R(1-\ol t_n)\ol {\mathcal{L}_1(t_n)}-\ol {\mathcal{L}_2(t_n)}}{\ol {\mathcal{L}_2(t_n)}- \chi_2(1-\ol t_n)^{-1}}=0.\eea    Hence we have $\lim_{\sigma\lra+\infty}\frac{\mathcal{L}_1-\mathcal{L}_2}{\mathcal{L}_2-1}=O\left(\frac{2}{3}\right)^{|\sigma|} $ as $\sigma\lra+\infty$.\par Therefore we have  $\frac{\mathcal{L}_1-\mathcal{L}_2}{\mathcal{L}_2-1}=O\left(\frac{2}{3}\right)^{|\sigma|}$, for some fixed $t=t_1$. Here it can be shown $H(\sigma+it_1)=O(e^{-\kappa|\sigma|})$, and hence by {\it{Lemma \ref{l2.6}}}, we have $\mathcal{L}_1=\mathcal{L}_2$.
	 
	\end{proof}
\begin{proof}[Proof of Theorem \ref{t1.2}]
	In this theorem we have considered $\mathcal{L}_1$, $\mathcal{L}_2$ are two zero degree $L$-function in $\mathcal{S}^{\#}$ with same functional equation, i.e., $\mathcal{L}_i(s)=Q^{1-2s}\ol {\mathcal{L}_i(1-\ol s)}$, for $i=1,2$ and $Q>1$ (see \cite{Kaczorowski_Perelli}).
	Since $\mathcal{L}_1$ and $\mathcal{L}_2$ share that $\mathcal{L}_1$ and $\mathcal{L}_2$ share the set $S=\{\alpha_1,\alpha_2,\ldots,\alpha_n\}$ CM, and without loss of generality assume $\alpha_1=1$. We can write it as \bea\label{e3.20} \frac{(\mathcal{L}_1-1)(\mathcal{L}_1-\alpha_2)\ldots(\mathcal{L}_1-\alpha_t) }{(\mathcal{L}_2-1)(\mathcal{L}_2-\alpha_2)\ldots(\mathcal{L}_2-\alpha_t) }=e^{as+b},\eea for some complex constants $a=a_1+ia_2$ and $b=b_1+ib_2$; $a_i,\;b_i\in\mathbb{R}$ for $i=1,2$.\\ Now \beas \left|\frac{(\mathcal{L}_1-1)(\mathcal{L}_1-\alpha_2)\ldots(\mathcal{L}_1-\alpha_t) }{(\mathcal{L}_2-1)(\mathcal{L}_2-\alpha_2)\ldots(\mathcal{L}_2-\alpha_t) }\right|=e^{a_1\sigma-a_2t}\\\left|\frac{(\ol {\mathcal{L}_1(1-\ol s)}-Q^{2s-1})\ldots(\ol {\mathcal{L}_1(1-\ol s)}-\alpha_nQ^{2s-1})}{(\ol {\mathcal{L}_2(1-\ol s)}-Q^{2s-1})\ldots(\ol {\mathcal{L}_1(1-\ol s)}-\alpha_nQ^{2s-1})}\right|=e^{a_1\sigma-a_2t},\eeas taking $\sigma\lra-\infty$ we have, $e^{a_1\sigma-a_2t}=1$, which forces $a_1=0$, and then we have $e^{-a_2t}=1\implies a_2=0$ also from (\ref{e3.20}), allowing $\sigma\lra-\infty$ we get $e^b=1$. \beas(\mathcal{L}_i-1)(\mathcal{L}_i-\alpha_2)\ldots(\mathcal{L}_i-\alpha_{n})=\mathcal{L}_i^t+a_1\mathcal{L}_i^{n-1}+a_2\mathcal{L}_i^{n-1}+\ldots+a_{n-1}\mathcal{L}_i+a_n,\eeas for $i=1,2$, where $a_j= (-1)^j \sum_{1\leq i_1<i_2<\ldots<i_j\leq n}\alpha_{i_1}
	\alpha_{i_2}\ldots \alpha_{i_j}\;\;\;j=1,2,\ldots,n$. 
	\par From (\ref{e3.20}) we have \bea\nonumber
	(\mathcal{L}_1^n-\mathcal{L}_2^n)+a_1(\mathcal{L}_1^{n-1}-\mathcal{L}_2^{n-1})+a_2(\mathcal{L}_1^{n-2}-\mathcal{L}_2^{n-2})+\ldots+a_{n-1}(\mathcal{L}_1-\mathcal{L}_2)&=&0,\eea 
	If $\mathcal{L}_1-\mathcal{L}_2\not\equiv0$ then from above we must have \bea\nonumber&&(\mathcal{L}_1^{n-1}+\mathcal{L}_1^{n-2}\mathcal{L}_2+\ldots+\mathcal{L}_2^{n-1})+a_{1}(\mathcal{L}_1^{n-2}+\mathcal{L}_1^{n-3}\mathcal{L}_2+\ldots+\mathcal{L}_2^{n-2})+\ldots+a_{n-1}\equiv0\\\nonumber&&\big({\ol {\mathcal{L}_1(1-\ol s)}}^{n-1}+{\ol {\mathcal{L}_1(1-\ol s)}}^{n-2}{\ol {\mathcal{L}_2(1-\ol s)}}+\ldots+{\ol {\mathcal{L}_2(1-\ol s)}}^{n-1}\big)+a_{1}Q^{(2s-1)}\big({\ol {\mathcal{L}_1(1-\ol s)}}^{n-2}+\ldots\\\label{e3.21}&+&{\ol {\mathcal{L}_2(1-\ol s)}}^{n-2}\big)+\ldots+a_{n-1}Q^{(n-1)(2s-1)}=0,
			\eea considering $\sigma\lra-\infty$ we have from (\ref{e3.21}) $n=0$, contradiction, therefore we must have $\mathcal{L}_1\equiv\mathcal{L}_2$.
\end{proof}



	
	{\section{Funding}}
	The authors have no funding.
	{\section{Disclosure Statement}}
	There exist no conflict of interest and competing interest whatsoever it maybe, between the two authors.


\begin{thebibliography}{90}
		
		\bibitem{Ban_Khoai_Kundu} A. Banerjee, H. H. Khoai and A. Kundu, Uniqueness of L-functions under sharing of sets, J. Theor. Nr. Bordx., $\mathbf{36(3)}$ (2024), 967-985.
		\bibitem{G.F_Springer(1977)} F. Gross, Factorization of meromorphic functions and some open problems Complex Analysis (Proc. Conf. Univ. Kentucky, Lexington, Kentucky, 1976), Lecture Notes in Math. $\mathbf{599}$, Springer-Berlin(1977), 51-69.
			\bibitem{Goenk-Haan-Ki_Math. Z.(14)} S. M. Gonek, J. Haan and H. Ki, A uniqueness theorem for functions in the extended Selberg class, Math. Z., ${\mathbf{278}}$(2014),
		995-1004.
		\bibitem{Hayman} W. K. Hayman, Meromorphic functions, Oxford University Press, Oxford (1964).
		\bibitem{Hu_LI_Can(16)} P. C. Hu and B. Q. Li, A simple proof and strengthening of a uniqueness theorem for $L$-functions, Can. Math. Bull., $\mathbf{59}$(2016), 119-122.
		\bibitem{Kaczorowski_Perelli} J. Kaczorowski and A. Perelli, On the structure of the Selberg class $0\leq d\leq1$, Acta Math.
		$\bf{182(2)}$(1999), 207-241.
		
		\bibitem{Khoai_An_Ramanujan.J} H. H. Khoai and V. H. An, Determining an $L$-function in the extended selberg class by it's preimages of subsets, The Ramanujan J., ${\mathbf{58}}$(2022), 253-267.
		\bibitem{H.Ki_Adv.Math(12)} H. Ki, A remark on the uniqueness of the Dirichlet series with a Riemann-type function
		equation, Adv. Math., $\mathbf{231(5)}$(2012), 2484-2490.
		\bibitem{Kundu_Banerjee_Comp.var} A. Kundu, A. Banerjee, L-functions and meromorphic functions satisfying the same Riemann type functional
		equation and sharing sets, Complex Var. Elliptic Equ., 70(4), 716–734 (doi.org/10.1080/17476933.2024.2338447).
\bibitem{Li_Adv_11}	B. Q. Li, A uniqueness theorem for Dirichlet series satisfying a Riemann type functional
	equation. - Adv. Math. $\mathbf{226(5)}$ (2011), 4198-4211.
		
		\bibitem{B.Q.LI-Proc.Am-10} B. Q. Li, A result on value distribution of $L$-functions, Proc. Am. Math. Soc., $\mathbf{138(6)}$(2010), 2071-2077.
		\bibitem{Li_Math.Z.} B. Q. Li, On the zeros of $L$-functions, Math. Z., $\mathbf{272}$(2012), 1097-1102.
		\bibitem{B.Q.Li_Proc.am(21)} B. Q. Li, On the number of zeros and poles of Dirichlet series, Trans. Amer. Math. Soc., $\mathbf{370}$(2018), 3865-3883.
		\bibitem{Lin-Lin-filomat}P. Lin and W. Lin,  Value distribution of $L$-functions concerning sharing sets, Filomat., $\mathbf{30}$(2016), 3795-3806.
		\bibitem{Selberg-92} A. Selberg, Old and new conjectures and results about a class of Dirichlet series, in: Proceedings of the Amalfi	Conference on Analytic Number Theory (Maiori, 1989), Univ. Salerno, Salerno (1992), 367-385.
		
		\bibitem{Steuding_Sprin(07)} J. Steuding, Value Distribution of $L$-Functions, Lect. Notes Math., Vol. 1877, Springer, Berlin(2007).
		\bibitem{Li_Du_Yi_Comp.var(22)} X. M. Li, X. R. Du and H. X. Yi, Dirichlet series satisfying a Riemann type functional equation
		and sharing one set, Complex Var. Elliptic Equ., ${\mathbf{68(10)}}$(2023), 1653-1677.
		\bibitem{C.C Yang-H.X.Yi}C. C. Yang and H. X. Yi, Uniqueness Theory of Meromorphic Functions, Math. Appl., 557, Kluwer Academic P., Dordrecht, 2003.
		
		\bibitem{Yan-Li-Yi_Lith} Q. Q.  Yuan, X. M.  Li, and H. X. Yi, Value distribution of $L$-functions and uniqueness questions of F. Gross, Lithuanian Math. J., $\mathbf{58(2)}$(2018), 249-262.	
		
		
		
		
	\end{thebibliography}
\end{document}